\documentclass[12pt, reqno]{amsart} 
\usepackage{amsmath, amsthm, amscd, amsfonts, amssymb, graphicx, color}
\usepackage[bookmarksnumbered, colorlinks, plainpages]{hyperref}
\input{mathrsfs.sty}
\newtheorem{theorem}{Theorem}[section]
\newtheorem{lemma}[theorem]{Lemma}
\newtheorem{proposition}[theorem]{Proposition}
\newtheorem{corollary}[theorem]{Corollary}
\theoremstyle{definition}
\newtheorem{definition}[theorem]{Definition}
\newtheorem{example}[theorem]{Example}

\theoremstyle{remark}
\newtheorem{remark}[theorem]{Remark}
\numberwithin{equation}{section}

\newcommand{\A}{\mathfrak{A}}

\newcommand{\E}{\mathscr{E}}
\newcommand{\F}{\mathscr{F}}

\newcommand{\HH}{\mathscr{H}}
\newcommand{\KK}{\mathscr{K}}
\newcommand{\CL}{\mathcal{L}}
\newcommand{\CC}{\mathbb{C}}

\begin{document}
\title[Vector-valued reproducing kernel Hilbert $C^*$-modules]{Vector-valued reproducing kernel Hilbert $C^*$-modules}
\author[M.S. Moslehian ]{Mohammad Sal Moslehian}
\address{Department of Pure Mathematics, Center of Excellence in Analysis on Algebraic Structures (CEAAS), Ferdowsi University of Mashhad, P. O. Box 1159, Mashhad 91775, Iran.}
\email{moslehian@um.ac.ir; moslehian@yahoo.com}

\renewcommand{\subjclassname}{\textup{2020} Mathematics Subject Classification}
\subjclass[]{46E22; 46L08; 47A56; 46L05.}

\keywords{Conditionally negative definite kernel; reproducing kernel Hilbert module; Hilbert $C^*$-module; Kolmogorov decomposition.}

\begin{abstract}
The aim of this paper is to present a unified framework in the setting of Hilbert $C^*$-modules for the scalar- and vector-valued reproducing kernel Hilbert spaces and $C^*$-valued reproducing kernel spaces. We investigate conditionally negative definite kernels with values in the $C^*$-algebra of adjointable operators acting on a Hilbert $C^*$-module. In addition, we show that there exists a two-sided connection between positive definite kernels and reproducing kernel Hilbert $C^*$-modules. Furthermore, we explore some conditions under which a function is in the reproducing kernel module and present an interpolation theorem. Moreover, we study some basic properties of the so-called relative reproducing kernel Hilbert $C^*$-modules and give a characterization of dual modules. Among other things, we prove that every conditionally negative definite kernel gives us a reproducing kernel Hilbert $C^*$-module and a certain map. Several examples illustrate our investigation.
\end{abstract} \maketitle

\section{Introduction}

The notion of Hilbert $C^{*}$-module is an extension of both those of Hilbert space and $C^*$-algebra. A Hilbert $C^*$-module over a $C^*$-algebra $\mathfrak{A}$ is a right $ \mathfrak{A}$-module $\mathscr E$ equipped with an $ \mathfrak{A}$-valued inner product $\langle \cdot, \cdot \rangle: \mathscr{H} \times \mathscr{H} \to \mathfrak{A}$ satisfying (i) $\langle x,x\rangle \geq 0$ with equality if and only if $x=0$, where $\geq$ denotes the usual partial order on the real space of all self-adjoint elements of $\A$; (ii) it is linear at the second variable; (iii) $\langle x,ya\rangle=\langle x,y\rangle a$; (iv) $\langle y,x\rangle=\langle x,y\rangle^*$ for all $x, y \in \E, a\in \A$, and $\lambda\in \mathbb{C}$, as well as, $\mathscr E$ endowed with the induced norm $\| x\| = \| \langle x,x\rangle \|^{\frac{1}{2}}\,\,(x\in \mathscr{H})$ is a complete space. When all conditions above except $\langle x,x\rangle=0 \Longrightarrow x=0$ and the completeness hold, we call $\E$ a semi-inner product $C^*$-module, which enjoys the Cauchy--Schwarz inequality
\begin{equation}\label{c-s}
\langle x,y\rangle \langle y,x\rangle \leq \|y\|^2\,\langle
x,x\rangle\qquad(x,y\in \E).
\end{equation}
In addition, $|x|=\langle x,x\rangle^{\frac{1}{2}}$ gives an $\A$-valued norm. For example, every $C^*$-algebra $\A$ is a Hilbert $C^*$-module over itself via $\langle a,b\rangle=a^*b,\,\, a,b\in\A$. By an $\A$-linear map between two Hilbert $C^*$-modules $\E$ and $\F$, we mean a linear map $f: \E\to \F$ satisfying the module property $f(xa)=f(x)a,\,\,a\in\A, x\in E$. A Hilbert $C^*$-module $\E$ is said to be \emph{self-dual} whenever for each bounded $\A$-linear map $f: \E \to \A$, there exists some $x_0\in\E$ such that $f(x)=\hat{x_0}(x):=\langle x_0,x\rangle$ for all $x\in \E$, in other words, the isometric $\A$-linear map $x\mapsto \hat{x}$ from $\E$ into the space $\E'$ consisting of all $\A$-module maps is surjective; see \cite[\S 2.5]{MT}. Note that $\E'$ is a right $\A$-module under $(\lambda f+g)(x)=\overline{\lambda}f(x)+g(x)$ and $(f\cdot a)(x)=a^*f(x)$. Both finitely generated Hilbert $C^*$-modules and Hilbert $C^*$-modules over finite-dimensional $C^*$-algebras are self-dual; see \cite{FRA}.

Hilbert $C^{*}$-modules behave in a similar manner as Hilbert spaces, nevertheless some basic Hilbert space properties such as the Riesz representation theorem, the adjointability of operators between Hilbert $C^*$-modules, and decomposition of the module into an orthogonal sum do not hold, in general. Recall that a closed submodule $\F$ of a Hilbert $C^*$-module is called \emph{orthogonally complemented} when $\E=\F\oplus \F^\perp$ in which $\F^\perp=\{x\in \E: \langle x,y\rangle=0 {\rm ~for~all~} y\in\F\}$.

Throughout the paper, all linear spaces are considered over the complex field $\mathbb{C}$, $\mathfrak{A}$ denotes a $C^*$-algebra, and capital letters in Euler fonts such as $\mathscr{E}$ stand for Hilbert $C^*$-modules over $\mathfrak{A}$ or, shortly, Hilbert $\A$-modules. Let $\mathcal{L}(\mathscr{E},\mathscr{F})$ denote the space of all operators $A:\mathscr{E}\to \mathscr{F}$ for which there exists an operator $A^*:\mathscr{F}\to \mathscr{E}$ satisfying $\langle Ax,y\rangle=\langle x,A^*y\rangle$ for any $x\in \mathscr{E}$ and $y\in \mathscr{F}$; such operators are called \emph{adjointable}. We write $\mathcal{L}(\mathscr{E}):=\mathcal{L}(\mathscr{E},\mathscr{E})$ that is a $C^*$-algebra endowed with the operator norm. An operator $A$ in the $C^*$-algebra $\mathcal{L}(\mathscr{E})$ is positive (see \cite[Lemma~4.1]{LAN}) if and only if 
\begin{align}\label{pos}
\langle Ax,x\rangle\geq0 \qquad{\rm ~for~all~} x\in \E.
\end{align} The reader is referred to \cite{LAN, MT} for detailed information about Hilbert $C^*$-modules.

The study of positive definite kernels goes back to the works of Hilbert around 1904, and conditionally positive definite kernels were first given by Schoenberg in 1940s. Aronszajn \cite{Aron} systematically developed the concept of reproducing kernel Hilbert space, which was introduced by Stanis\l aw Zaremba in 1907. Roughly speaking, these are spaces of functions in which when two functions are close in the norm, then they are pointwise close. There exists a one-to-one correspondence between positive definite kernels and reproducing kernel Hilbert spaces. The (Kolmogorov) decomposition of positive definite kernels was explored by Kolmogorov in the framework of Hilbert spaces. Murphy \cite{MUR} introduced the Kolmogorov decomposition of positive definite kernels in the framework of Hilbert $C^*$-modules. This decomposition was explored by the author \cite{MOS1} for conditionally positive definite kernels. Reproducing kernel Hilbert spaces, positive definite kernels, and their Kolmogorov decompositions, especially in the context of holomorphic functions (see \cite{CAR, KUM}), have been extended to other settings \cite{MUR, BAA, SZ2} and have appeared in many applications in several disciplines such as machine learning, integral equations, complex analysis, and probability; see \cite{BBLS, BER, An} and references therein. Another related concept is that of relative reproducing kernel Hilbert space. In \cite{ALP}, the authors investigated such spaces and associate with any map satisfying some conditions, a reproducing kernel Hilbert space over the field of real numbers whose reproducing kernel is the given map.

This paper is organized as follows. In Section 2, we investigate conditionally negative definite kernels with values in the $C^*$-algebra $\CL(\E)$, provided some examples, and state their one-to-one correspondence with positive definite kernels. In Section 3, we first show that there is a two-sided connection between positive definite kernels and reproducing kernel Hilbert $C^*$-modules. Furthermore, we explore some conditions under which a function is in the reproducing kernel module and present an interpolation theorem. Section 4 starts with a review of some basic properties of the so-called relative reproducing kernel Hilbert $C^*$-modules. We then provide a characterization of such modules. Every reproducing kernel Hilbert $C^*$-module is evidently a relative reproducing kernel Hilbert $C^*$-module while the reverse is not true, in general. Nevertheless, we prove that the reverse statement holds under some mild conditions. Finally, we show that every conditionally negative definite kernel gives us a reproducing kernel Hilbert $C^*$-module and a certain map. 

The paper enriches the theory of conditionally negative definite kernels and relative reproducing kernel Hilbert spaces in two directions while we model some ideas from \cite{ALP, MOS1} to our setting: One direction is toward a general framework including all other known settings such as scalar-, vector-, and $C^*$-valued kernels. The other one shows that conditionally negative definite kernels give reproducing kernel Hilbert spaces, in general, and only in some special cases, they give relative reproducing kernel Hilbert spaces.


\section{Positive definite and conditionally negative definite kernels}

We mimic some techniques of the scalar theory of conditionally negative definite kernels. This extension covers all scalar, operator,  and $C^*$-valued versions of conditionally negative definite kernels. Let us begin our investigation with the following definition.

\begin{definition}
A matrix $A=[a_{ij}]$ in the set $\mathbb{M}_n(\A)$ of $n \times n$ matrices with entries in $\A$ is called \emph{positive definite} if it is positive in the $C^*$-algebra $\mathbb{M}_n(\A)$ or, equivalently, if and only if $\sum_{i,j=1}^na_i^*a_{ij}a_j \geq0$ for all $a_1, \ldots, a_n \in\A$. 

A hermitian matrix $A=[a_{ij}]\in \mathbb{M}_n(\A)\,\, (n\geq 2)$ is called \emph{conditionally negative definite} if $\sum_{i,j=1}^na_i^*a_{ij}a_j \leq0$ for all $a_1, \ldots, a_n \in\A$ with $\sum_{i=1}^na_i=0$. 
\end{definition}
\begin{example}\label{ex1}
\begin{itemize}
\item[(i)] Any sum of matrices of the form $[c_i^*c_j]$ for $c_1, \ldots, c_n \in \A$ is positive definite and vice versa; see \cite[Lemma IV.3.2]{TAK}.
\item[(ii)] If $x_1, \ldots, x_n$ are in a semi-inner product $C^*$-module, then $[\langle x_i,x_j\rangle]$ is positive definite since, by the previous example, we have
\begin{align*}
\sum_{i,j=1}^n a_i^*\langle x_i,x_j\rangle a_j=\left\langle \sum_{i=1}^nx_ia_i, \sum_{i=1}^nx_ia_i\right\rangle\geq 0,\quad(a_1,\ldots,a_n\in\A).
\end{align*}
\item[(iii)] A matrix of the form $[c_i+c_j^*]$ $c_1, \ldots, c_n \in \A$ is clearly conditionally negative definite.
\item[(iv)] Any matrix of the form $[|x_i-x_j|^2]$ is conditionally negative definite for all $x_1,\ldots, x_n$ in a Hilbert $C^*$-module $\E$ subject to the condition that all $\langle x_i,x_j\rangle,\,\, 1\leq i,j\leq n$ are self-adjoint in $\A$. Note that
\begin{align*}
|x_i-x_j|^2=\langle x_i-x_j,x_i-x_j\rangle=|x_i|^2+|x_j|^2-2\langle x_i,x_j\rangle.
\end{align*}
Therefore,
\begin{align*}
\sum_{1\leq i,j\leq n}&a_i^*|x_i-x_j|^2a_j\\&=\sum_{i=1}^na_i^*|x_i|^2\sum_{j=1}^na_j+\sum_{i=1}^na_i^*\sum_{j=1}^n|x_j|^2a_j-2\left\langle \sum_{i=1}^nx_ia_i,\sum_{j=1}^nx_ja_j\right\rangle\\&=-2\left\langle \sum_{i=1}^nx_ia_i,\sum_{j=1}^nx_ja_j\right\rangle\leq 0.
\end{align*}

\end{itemize}
\end{example}

The concept of conditionally negative definite kernel has been used in significant fields including harmonic analysis and probability theory; see \cite{BER}.

A \emph{kernel} is a map $K: S\times S \to \mathcal{L}(\E)$ in which $S$ is a nonempty set. A kernel $K$ is called \emph{hermitian} if $K(t,s)=K(s,t)^*$ for all $s,t \in S$, and it said to be \emph{symmetric} if $K(t,s)=K(s,t)$ for all $s,t \in S$.
\begin{definition}
If a (hermitian) map $L:S \times S \to \mathcal{L}(\E)$ satisfies
\begin{eqnarray}\label{pdk1}
\sum_{i, j=1}^n \langle L(s_i,s_j)x_i, x_j\rangle \geq 0,
\end{eqnarray}
for all $n\geq 2$, all $s_1, \ldots, s_n\in S$, and all $x_1, \ldots, x_n\in\E$, then it said to be \emph{positive definite}. 
\end{definition}
\begin{example}\label{ex2}
\begin{itemize}
\item[(i)]Given a function $f:S\to \E$, the kernel $K:S\times S\to \CL(\E)$ defined by $K(s,t)=f(t)\otimes f(s)$, where $(x\otimes y)(z):=x\langle y,z\rangle,\,\, x,y,z\in \E$, is positive definite, since
\begin{align*}
\sum_{i,j=1}^n\langle K(s_i,s_j)x_i,x_j\rangle&=\sum_{i,j=1}^n \big\langle f(s_j)\langle f(s_i), x_i\rangle,x_j\big\rangle\\
&=\sum_{i,j=1}^n \langle f(s_i),x_i\rangle^* \langle f(s_j),x_j\rangle\\
&=\left(\sum_{i=1}^n \langle f(s_i),x_i\rangle\right)^*\left(\sum_{j=1}^n \langle f(s_j),x_j\rangle\right)\geq 0.
\end{align*}
\item[(ii)] Suppose that $\Delta$ is a compact Hausdorff space equipped with a Radon measure $\mu$ and $\A$ is a $C^*$-algebra of operators that act on a Hilbert space. By  \emph{continuous field of operators} we mean a family $(a_t)_{t\in \Delta}$ of operators in ${\mathscr A}$ such that the function $t \mapsto a_t$ is continuous on $\Delta$ in the norm topology and $t \mapsto \|a_t\|$ is integrable. Then the Bochner integral
$\int_{\Delta}a_t{\rm d}\mu(t)$ is the unique element in
$\A$ satisfying
$\varphi\left(\int_\Delta a_t{\rm d}\mu(t)\right)=\int_\Delta \varphi(a_t){\rm d}\mu(t)$
for all linear functionals $\varphi$ in the norm dual of $\A$; see \cite{MB} for the detailed construction of the Bochner integral. Let $\kappa$ be a scalar-valued positive definite kernel on $\Omega$. Then 
\[K(s,t)=\int_\Delta\int_\Delta \kappa(s,t)a_s^*a_t{\rm d}\mu(s){\rm d}\mu(t)\]
is a positive definite kernel due to the fact that the product of two positive definite kernels $\kappa(s,t)$ and $[a_s^*a_t]$ is positive definite.
\end{itemize} 
\end{example}
A related concept reads as follows.
\begin{definition}
A hermitian kernel $K:S \times S \to \mathcal{L}(\E)$ is called \emph{conditionally negative definite} whenever
\begin{eqnarray}\label{pdk2}
\sum_{i, j=1}^n \langle K(s_i,s_j)x_i,x_j\rangle \leq 0,
\end{eqnarray}
for all $n\geq 2$, all $s_1, \ldots, s_n\in S$, and all $x_1, \ldots, x_n\in\E$ with $\sum_{i=1}^nx_i=0$.
It said to be normalized if $K(s,s) = 0$ for all $x \in S$.
\end{definition}
Passing to the unitization of $\A$ if necessary, assume that $\A$ is unital. Therefore, when $\E=\A$, we get $\mathcal{L}(\E)=\A$ since we can identify $t\in\CL(\A)$ with the left multiplication by $t(1)$ due to $l _{t(1)}(a)=t(1)a=t(a),\,\,a\in\A$. Therefore, inequalities \eqref{pdk1} and \eqref{pdk2} entail that a kernel $K$ is conditionally negative definite (resp., positive definite) if and only if so is the matrix $[K(s_i,s_j)]$ in $\mathbb{M}_n(\A)$ for all $s_1, \ldots, s_n\in S$. If $\E$ is a Hilbert space and $\A$ is the field $\mathbb{C}$ of scalars ($\E=\A=\mathbb{C}$, resp.), we arrive at the so-called vector-valued kernels (scalar-valued kernels, resp.) by considering left $\A$-modules.

The following result is known for scalar kernels (see \cite{BER}) and Hilbert $C^*$-modules (see \cite[Theorem 2.4]{MOS1}). Its proof is straightforward, and we omit it. 

\begin{theorem}\label{th1}
Let $L: S\times S \to \CL(\E)$ be a hermitian kernel and let $s_0\in S$ be arbitrary. Then $L$ is conditionally negative definite if and only if the kernel $K: S\times S \to \CL(\E)$ defined by
\begin{align}\label{lk}
K(s,t):=\frac{1}{2}\left[L(s,s_0) + L(s_0,t)-L(s,t)-L(s_0,s_0)\right]
\end{align}
is positive definite.
\end{theorem}


\section{Structure of reproducing kernel Hilbert $C^*$-modules}

In this section, the first assertion provides a key theorem in a general form. Note that the set $\mathcal{S}(S,\E)$ of all maps from $S$ into $\E$ can be endowed with a right $\A$-module structure via the usual addition and scalar multiplication on functions and the module action $(f\cdot a)(s)=f(s)a,\,\,a\in\A, s\in S, f\in \mathcal{S}(S,\E)$.

\begin{theorem}\label{main1}
Let $S$ be a nonempty set, let $\E$ be a Hilbert $C^*$-module over a $C^*$-algebra $\A$, and let $K: S\times S\to \CL(\E)$ be a kernel. Then the following assertions are equivalent:
\begin{itemize}
\item[(i)] The kernel $K$ is positive definite.
\item[(ii)] There is a Hilbert $\A$-module $\HH$ consisting of $\E$-valued functions on $S$ having the following properties:
\begin{itemize}
\item[(1)] For every $x\in \E$ and every $s\in S$, the function $K_{x,s}: S\to \E$ defined by
\begin{align}\label{kxs}
K_{x,s}(t):=K(s,t)(x),\quad t \in S,
\end{align}
is in $\HH$, and the $\A$-linear span of $\{K_{x,s}: x\in\E, s\in S\}$ is dense in $\HH$.
\item[(2)] The $\A$-valued inner product on $\HH$ is defined via
\begin{align}\label{kk}
\langle K_{x,s},K_{y,t}\rangle:=\langle K_{x,s}(t),y\rangle=\langle K(s,t)(x), y\rangle,\,\, x,y\in\E, s,t\in S
\end{align}
and extended to whole $\HH$ by linearity. Indeed, for all $f\in \HH$, $x\in\E$, and $s\in S$, it holds that
\begin{align}\label{rk}
\langle f(s),x\rangle=\langle f,K_{x,s}\rangle.
\end{align}
\end{itemize}
\item[(iii)] There are a Hilbert $\A$-module $\HH$ and a feature map $\delta:S\to \CL(\HH,\E),\, s\mapsto \delta_s$, where $\delta_s(f):=f(s), f\in\HH$ is the evaluation map at $s$, such that \begin{align}\label{delta}
K(s,t)=\delta_t\delta_s^*.
\end{align} 
\end{itemize}
\end{theorem}
\begin{proof}
We prove only \eqref{rk} and \eqref{delta}. The rest can be verified by the same method as employed in the construction of reproducing kernel Hilbert spaces (see \cite{MOS1, MUR, SZ}).

To see \eqref{rk}, assume $f=K_{y,t}$. Then
\[\langle f,K_{x,s}\rangle=\langle K(t,s)(y), x\rangle= \langle f(s),x\rangle.\]
The required equality can be deduced by passing to the $\A$-linear span of $K_{y,t}$'s.

To establish \eqref{delta}, we note that the adjoint operator of $\delta_s$ is $\delta_s^*(x)=K_{x,s}$ for all $s\in S$ and $x\in \E$. It follows from \eqref{kxs} that
\[(\delta_t\delta_s^*)(x)=\delta_t(K_{x,s})=K_{x,s}(t)=K(s,t)(x).\]
\end{proof}
The Hilbert $C^*$-module $\HH$ in Theorem \ref{main1} is called the \emph{reproducing kernel Hilbert $C^*$-module} associated with $K$. Such structures are studied in some kinds of Stinespring theorems in other settings \cite{ACM, BBLS}.
The decomposition $K(s,t)=\delta_s^*\delta_t$ is called the \emph{Kolmogorov decomposition} of $K$. 
\begin{remark}\label{sum}
\begin{itemize}
\item[(i)] Because of
\begin{align}\label{ka}
\left(K_{x,s}\cdot a\right)(t)=K(s,t)(x)\cdot a= K(s,t)(x\cdot a)=K_{x\cdot a,s}(t),\quad a\in\A,
\end{align}
the $\A$-linear span of $\{K_{x,s}: x\in\E, s\in S\}$ is nothing more than its linear spanned subspace. 
\item[(ii)] If a sequence $(f_n)$ is convergent to $f$ in $\HH$, then it follows from \eqref{rk} that $\langle f_n(s),x\rangle \to \langle f(s),x\rangle$ for all $s\in S$ and $x\in\E$; in particular, in the case of classical Hilbert spaces, $(f_n(s))$ is a convergent sequence in $\mathbb{C}$ and the convergence is uniform on any set $\{s\in S: K(s,s)\leq M\}$, where $M>0$.
\item[(iii)] Motivated by \eqref{kk} and identifying $\CL(\A)$ with $\A$, one can define a positive definite kernel $\widetilde{K}$ on the set $\E\times S$ into $\A$ by $\widetilde{K}((x,s),(y,t))=\langle K(s,t)x,y\rangle$. If $\widetilde{\HH}$ is the corresponding positive definite kernel, then for every $\widetilde{f}\in\widetilde{\HH}$, there is a function $f\in\HH$ such that $\widetilde{f}(x,t)=\langle f(s),x\rangle$. To show this, it is enough to consider $\widetilde{f}$ to be the basic function $\widetilde{K}_{y,t}$ and observe that $K_{y,t}$ is the required function in $\HH$, since
\[\widetilde{K}_{y,t}(x,s)=\widetilde{K}((x,s),(y,t))=\langle K(s,t)x,y\rangle=\langle K_{y,t}(s),x\rangle.\]
\end{itemize}
\end{remark}
\begin{example}
Let $f: S\to \E$ be a nonzero function and let $K(s,t)=f(t)\otimes f(s)$. Example \ref{ex2}(i) shows that $K$ is a positive definite kernel. Its corresponding reproducing kernel Hilbert $C^*$-module is the closure of the $\A$-module generated by $f$. To show this, note that 
\[K_{x,s}(t)=K(s,t)(x)=(f(t)\otimes f(s))(x)=f(t)\langle f(s),x\rangle=(f\cdot \langle f(s),x\rangle)(t) ,\quad t\in S\]
whence
\[K_{x,s}=f\cdot\langle f(s),x\rangle =f\cdot \langle f,K_{x,s}\rangle=(f\otimes f)(K_{x,s}),\quad x\in\E, s\in S.\]
Due to the fact that $\A$-linear span of $K_{x,s}$'s is dense in $\HH$, we arrive at $f\otimes f=I_\E$, where $I_\E$ denotes the identity operator on $\E$. Hence the norm of $f$ in $\HH$ is one since $\|f\otimes f\|=\|f\|\,\|f\|$.

In the scalar case where $\A=\CC$ and $\E$ is a Hilbert space, the kernel, by taking the left module structure into account, turns into $K(s,t)=f(s)\overline{f(t)}$ and the reproducing kernel Hilbert space $\HH$ is the subspace spanned by $f$.
\end{example}

\begin{example}
Let $\F$ be a Hilbert $C^*$-module over a unital $C^*$-algebra $\A$ with unit $1$. Let $S=\F$ as a set and let $K(s,t)=l_{\langle t,s\rangle}$ as a kernel on $S$ into $\CL(\A)$, where for $a\in\A$, the operator $l_a\in\CL(\A)$ denotes the left multiplication $l_a(b)=ab,\,\,b\in\A$. The Gram matrix $[\langle s_i,s_j\rangle]_{1\leq i,j\leq n}$ is positive definite for all $s_1, \ldots, s_n\in S$ and so is its transpose. Hence, the kernel $K$ is positive definite. From the construction stated at Theorem \ref{main1}, we have
\[K_{a,s}(t)=K(s,t)(a)=l_{\langle t,s\rangle}(a)=\langle t,s\rangle a=\langle t,as\rangle.\]
Hence, $K_{a,s}=\langle \cdot, as\rangle$. Since $A$ is unital, the reproducing kernel Hilbert $C^*$-module associated with $K$ is the closed $\A$-linear span of $\{f_s:=K_{1,s}=\langle \cdot,s\rangle: s\in S\}$ endowed with the inner product 
\[\langle f_s,f_t\rangle=\langle K_{1,s},K_{1,t}\rangle=\langle K(s,t)(1), 1\rangle=\big(K(s,t)(1)\big)^*1=\big(l_{\langle t,s\rangle}(1)\big)^*=\langle s,t\rangle.\]
Moreover,
\[\langle f_t(s),a\rangle=K_{1,t}(s)^*a=\big(K(t,s)(1)\big)^*a=\langle K_{1,t},K_{a,s}\rangle=\langle f_t,K_{a,s}\rangle,\]
from which we get $\langle f(s),a\rangle=\langle f,K_{a,s}\rangle$.
It is remarkable that the Riesz representation theorem is valid for $\CL(\F,\A)$ since $\A$ is unital; see \cite[p. 13]{LAN}. In the case when $\A=\CC$ and $\F$ is a Hilbert space, the above reproducing kernel Hilbert space is the space of all bounded linear functionals on $\F$.
\end{example}

\begin{example}
Suppose that $K_0: S\times S\to \mathbb{C}$ is a positive definite kernel with the associated reproducing kernel Hilbert space $\mathcal{H}$ and that $\E$ is an arbitrary Hilbert $C^*$-module. Then $K: S\times S\to \CL(\E)$ defined by $K(s,t)=K_0(s,t)I_\E$ is trivially a positive definite kernel. In virtue of \eqref{kk}, we have
\[\langle K_{x,s},K_{y,t}\rangle=\langle K(s,t)(x), y\rangle=K_0(s,t)\langle x,y\rangle=\langle K_{0_{s}},K_{0_{t}}\rangle\langle x,y\rangle,\,\, x\in\E, s\in S ,\]
from which we conclude that the reproducing kernel Hilbert $C^*$-module associated with $K$ is $\HH={\mathcal E}\otimes \E$ (see \cite[p. 6]{LAN}) equipped with the $\A$-valued inner product defined by
\[\langle\xi\otimes x,\eta\otimes y\rangle:=\langle \xi,\eta\rangle\langle x,y\rangle.\]
\end{example}
\begin{proposition}\label{main3}
Let $\HH$ be a reproducing kernel Hilbert $C^*$-module with kernel $K: S\times S\to \CL(\E)$ and let $f\in\HH$. Then the kernel $\|f\|K(s,t)-f(t)\otimes f(s)$ is positive definite.
\end{proposition}
\begin{proof}
Applying the Cauchy--Schwarz inequality \eqref{c-s}, one can define a semi-inner product on $\HH$ by
\begin{align*} 
\langle\cdot,\cdot\rangle_f: \HH \times \HH\to\A,\quad \langle g,h\rangle_f:= \|f\|^2\,\langle g,h\rangle-\langle g,f\rangle \langle f,h\rangle.
\end{align*}
Evidently, $\langle\cdot,\cdot\rangle_f$ is a semi-inner product on $\HH$; see \cite{ARA}. Applying Example \ref{ex1}(ii) to $\langle\cdot,\cdot\rangle_f$ yields that $[\langle g_i,g_j\rangle_f]\geq 0$ in $\mathbb{M}_n(\A)$, which entails that 
\begin{align}\label{bak}
\|f\|^2\big[\langle g_i, g_j\rangle\big]\ge \big[\langle g_i,f\rangle\langle f,g_j\rangle\big]
\end{align}
for all $g_1,\ldots, g_n\in\HH$. 

Next, let $s_1,\ldots, s_n\in S$ and let $x_1, \ldots, x_n\in \E$. We have
\begin{align*}
\big[\langle (f(s_j)\otimes f(s_i))x_i,x_j\rangle\big]&=\big[\langle f(s_j)\langle f(s_i),x_i\rangle ,x_j\rangle\big]=\big[\langle x_i,f(s_i)\rangle\langle f(s_j),x_j\rangle\big]\\
&=\big[\langle K_{x_i,s_i},f\rangle\langle f,K_{x_j,s_j}\rangle\big]\leq \big[\|f\|\langle K_{x_i,s_i},K_{x_j,s_j}\rangle\big]\\
&=\big[\|f\|\langle K(s_i,s_j)x_i,x_j\rangle\big].\\
\end{align*}
Thus the kernel $\|f\|K(s,t)-f(t)\otimes f(s)$ corresponded to the positive definite matrix 
$\big[\|f\|\langle \big(K(s_i,s_j) -f(s_j)\otimes f(s_i)\big)x_i,x_j\rangle$
is positive definite.
\end{proof}
\begin{remark}
The converse of Proposition \ref{main3} is true in the set-up of Hilbert spaces, in this form that if $f: S\to \CC$ is a function such that, for some $m>0$ depending on $f$, we have
\[\left|\sum_{i=1}^nf(s_i)\lambda_i\right|^2\leq m^2\sum_{i,j=1}^nK(s_i,s_j)\lambda_i\overline{\lambda_j},\quad s_1, \ldots, s_n\in S, \lambda_1, \ldots, \lambda_n\in \CC,\]
then $f\in \HH$ and $\|f\|$ is equal to the smallest $m$ satisfying the above inequality. This property is called the \emph{RKHS test}. It is not known to the author whether the converse of 
Proposition \ref{main3} is true without any self-duality assumption; see \cite[Proposition 4]{SZ}.
\end{remark}
To achieve the next result, we need a version of the Douglas lemma in the setup of Hilbert $C^*$-modules; see also \cite{MMX}.

\begin{lemma}\cite[Theorem 3.2]{FAN} \label{fang} Let $\E$ and $\KK$ be Hilbert $\mathfrak{A}$-modules and let $B$ be in $\mathcal{L}(\E,\KK)$. Then the following statements are equivalent:
\begin{enumerate}
\item[{\rm (i)}] $\overline{{\rm ran}(B^*)}$ is orthogonally complemented;
\item[{\rm (ii)}] For any Hilbert $\mathfrak{A}$-module $\F$ and any $A\in {\mathcal L}(\F,\KK)$, the equation
$A =BX$ for $X\in {\mathcal L}(\F,\E)$ is solvable whenever ${\rm ran}(A)\subseteq {\rm ran}(B)$;
\end{enumerate}
If condition {\rm (i)} is fulfilled and $A\in {\mathcal L}(\F,\KK)$ is such that ${\rm ran}(A)\subseteq {\rm ran}(B)$, then there exists a unique $C\in {\mathcal L}(\F,\E)$ satisfying
\[A=BC\quad \mbox{and}\quad {\rm ran}(C)\subseteq \ker(B)^\perp.\]
In this case, $\|C\|^2=\inf\left\{\lambda : AA^*\leq \lambda BB^*\right\}$.
\end{lemma}
It is easy to verify that the right $\A$-module $\E^n$ can be equipped with a $C^*$-inner product via 
\[\langle (x_1, \ldots, x_n),(y_1, \ldots, y_n)\rangle=\sum_{i=1}^n\langle x_i,y_i\rangle.\]
In addition, the $C^*$-algebra $\CL(\E^n)$ can be naturally identified with the $C^*$-algebra $\mathbb{M}_n(\CL(\E))$. Hence, when we are dealing with the action of a matrix of operators in $\mathbb{M}_n(\CL(\E))$ on $\E^n$, we represent the elements of $\E^n$ in the form of a column vector $(x_1, \ldots, x_n)^t$.

\begin{theorem}\label{main4}
Let $\HH$ be a reproducing kernel Hilbert $C^*$-module over a unital $C^*$-algebra $\A$ with kernel $K: S\times S\to \CL(\E)$ and let $f: S\to \E$ be a function. If the kernel $mK(s,t)-f(t)\otimes f(s)$ is positive definite for some $m>0$, then for every $s_1,\ldots, s_n\in S$, there is a function $h\in \HH$ such that 
$\|h\|\leq m$ and $f(s_i) = h (s_i)$ for $i=1, \ldots, n$.
\end{theorem}
\begin{proof} Suppose that $s_1,\ldots, s_n\in S$. Let $A: \A\to \E^n$ be  defined by \[A(a)=(f(s_1)a, \ldots, f(s_n)a).\]
Then $A^*:\E^n \to \A$ takes $(x_1, \ldots, x_n)\in \E^n$ to $\sum_{i=1}^n\langle f(s_i),x_i\rangle$ since $\langle As,x\rangle=\sum_{i=1}^n\langle f(s_i)a,x_i\rangle=\langle a,A^*x\rangle$.
In addition,
\begin{align*}
\left\langle AA^*(x_1, \ldots, x_n)^t, (x_1, \ldots, x_n)^t\right\rangle&=\left(\sum_{j=1}^n\langle f(s_j),x_j\rangle\right)^*\left(\sum_{i=1}^n\langle f(s_i),x_i\rangle\right)\\
&= \sum_{i,j=1}^n\langle x_j,f(s_j)\rangle\langle f(s_i),x_i\rangle\\
&=\sum_{i,j=1}^n\langle (f(s_i)\otimes f(s_j))x_j,x_i\rangle\\
&=\left\langle \Big(f(s_i)\otimes f(s_j)\Big) (x_1, \ldots, x_n)^t, (x_1, \ldots, x_n)^t\right\rangle.
\end{align*}
Passing to the equivalent definition of positive definiteness of kernels in terms of matrices and using Theorem \ref{main1}(iii), we get
\[AA^*\leq \Big(f(s_i)\otimes f(s_j)\Big) \leq m\Big(K(s_j,s_i)\Big)\leq m\Big(\delta_{s_i}\delta_{s_j}^*\Big)=mBB^*\]
in $\mathbb{M}_n(\CL(\E))$, where $B=\big(\delta_{s_1},\ldots, \delta_{s_n}\big)\in \CL(\HH,\E^n)$.
Furthermore, ${\rm ran}(A) \subseteq {\rm ran}(B)$ since
\[ A(a)=(f(s_1)a, \ldots, f(s_n)a)=((f\cdot a)(s_1), \ldots, (f\cdot a)(s_n))\in {\rm ran}(B).\]
Moreover, in virtue of Theorem \ref{main1},
\begin{align}\label{ka}
\left(\lambda K_{x,s}\right)(t)=\lambda K(s,t)(x)= K(s,t)(\lambda x)=K_{\lambda x,s}(t),\quad \lambda\in\CC.
\end{align}
Hence the linear span of $\{K_{x,s}: x\in\E, s\in S\}$ is nothing more than the set of all finite sums of elements of this set. Thus $\overline{{\rm ran}(B^*)}$ is the closure of $\{\sum_{i=1}^nK_{x_i,s_i}: x_i\in \E, s_i\in S, 1\leq i\leq n, n\in\mathbb{N}\}$ that is $\HH$. Applying Lemma \ref{fang}, we infer that there is a unique operator $C\in\CL(\A,\HH)$ such that ${\rm ran}(C)\subseteq \ker(B)^\perp=\{g\in\HH: g(s_1)=\dots=g(s_n)=0\}^\perp$, $\|C\|\leq m$, and $A=BC$. Therefore, if $h=C(1)$, then
\[(f(s_1), \ldots, f(s_n))=A(1)=BC(1)=(\delta_1(h), \ldots, \delta_n(h))=(h(s_1), \ldots, h(s_n)),\]
whence
$f(s_i)=h(s_i)$ for all $i\in\{1, \ldots, n\}$.
\end{proof}
We denote by $\mathbb{K}(\mathcal{H})$ the $C^*$-algebra of all compact operators acting on a Hilbert space $\mathcal{H}$. It is known from \cite[Theorem 1.4.5]{ARV} that a $C^*$-algebra $\A$ of (not necessarily all) compact operators is a $c_0$-direct sum of $C^*$-algebras of the form $\mathbb{K}(\mathcal{H}_i)$, in other words, $\A=c_0$-$\oplus_i\mathbb{K}(\mathcal{H}_i)$. Such $C^*$-algebras enjoy many significant properties. One of them is stated below. 
\begin{lemma} \cite{MAG} \label{c1} Let $\A$ be a $C^*$-algebra. The following conditions are equivalent:
\begin{itemize}
\item[(i)] $\A=c_0$-$\oplus_i\mathbb{K}(\mathcal{H}_i)$, that is, it has a faithful $*$-representation
as a $C^*$-algebra of compact operators on some Hilbert space.
\item[(ii)] For every Hilbert $\A$-module $\E$, every Hilbert $\A$-submodule is automatically orthogonally complemented in $\E$.
\end{itemize}
\end{lemma}
To establish the next result, we need the following lemma.
\begin{lemma}\cite[Theorem 2.4]{MOU}\label{c2}
Let $A, C\in \mathcal{L}(\mathscr{E})$. Suppose that $\overline{{\rm ran}(A^*)}$ and $\overline{{\rm ran}(B^*)}$ are orthogonally complemented in $\mathscr{E}$. If $AA^*=\lambda BB^*$ for some $\lambda >0$, then ${\rm ran}(A)={\rm ran}(B)$.
\end{lemma}
We are ready to prove an interpolation theorem. It is completely differs from another one for usual reproducing kernel Hilbert $C^*$-modules; see \cite[Theorem 2.7]{GMM}
\begin{theorem}\label{main5}
Let $\HH$ be a reproducing kernel Hilbert $C^*$-module over a $C^*$-algebra $\A$ of compact operators with kernel $K: S\times S\to \CL(\E)$, let $s_1,\ldots, s_n\in S$ be distinct, and let $y_1,\ldots, y_n\in \E$. There exists a function $f\in \HH$ with $f(s_i)=y_i,\,\,1\leq i\leq n$ if and only if $(y_1,\ldots, y_n)^t$ belongs to the range of operator matrix $\Big(K(s_j,s_i)\Big)$.
\end{theorem}
\begin{proof}
Let $F=\{s_1,\ldots, s_n\}$. Consider the multi-evaluation map $\delta_F: \HH\to \E^n$ defined by 
\[\delta_F(f)=(f(s_1), \ldots, f(s_n)),\quad f\in\HH.\]
It is easy to see that $\delta_F^*(x_1, \ldots, x_n)=\sum_{i=1}^nK_{x_i,s_i}$. Then
\begin{align*}
(\delta_F\delta_F^*)(x_1, \ldots, x_n)^t&=\delta_F\left(\sum_{i=1}^nK_{x_i,s_i}\right)=\left(\sum_{i=1}^nK_{x_i,s_i}(s_1), \ldots, \sum_{i=1}^nK_{x_i,s_i}(s_n)\right)\\
&=\Big(K(s_j,s_i)\Big)(x_1,\ldots, x_n)^t.
\end{align*}
Thus $\Big(K(s_j,s_i)\Big)^{1/2}\Big(K(s_j,s_i)\Big)^{1/2}=\delta_F\delta_F^*$. It follows from Lemma \ref{c1} that the closures of the submodules ${\rm ran}(\delta_F)$ and ${\rm ran}\Big(K(s_j,s_i)\Big)^{1/2}$ are orthogonally complemented, and so by Lemma \ref{c2}, 
${\rm ran}(\delta_F)={\rm ran}\Big(K(s_j,s_i)\Big)^{1/2}$ as required.
\end{proof}


\section{Relative reproducing kernel Hilbert $C^*$-modules}

We aim to show that conditionally negative definite kernels give reproducing kernel Hilbert
spaces, in general, and only in some special cases they give relative reproducing
kernel Hilbert spaces. We begin this section with following definition inspired by \cite{JOR}. 
\begin{definition}
A Hilbert $C^*$-module $\HH$ of $\E$-valued functions on a set $S$ is said to be a \emph{relative reproducing kernel Hilbert $C^*$-module} if there is a map $D: \E\times S\times S \to \HH,\,\, (x,s,t)\mapsto D_{x,s,t}$ such that 
\begin{eqnarray}\label{rrk}
\langle f(s)-f(t),x\rangle= \langle f, D_{x,s,t}\rangle,\qquad f\in\HH, x\in\E, s,t\in S.
\end{eqnarray}
We call $D$ the \emph{relative reproducing kernel}.
\end{definition}
\begin{remark} \label{iff} We would like to give some straightforward comments.
\begin{itemize}
\item[(i)] It follows from \eqref{rrk} that $f$ is in the orthogonal complement of the set $\{D_{x,s,t}: x\in\E, s,t\in S\}$ if and only if $f$ is a constant function.
\item[(ii)] Let $h_{x,s}:=D_{x,s,s_0}$ for some fixed element $s_0\in S$. For all $f\in \HH$, we have
\begin{align*}
\langle f, D_{x,s,t}\rangle &=\langle f(s)-f(t),x\rangle=\langle f(s)-f(s_0),x\rangle-\langle f(t)-f(s_0),x\rangle\\
& =\langle f, D_{x,s,s_0}\rangle -\langle f, D_{x,t,s_0}\rangle=\langle f, h_{x,s}-h_{x,t}\rangle,
\end{align*}
whence 
\begin{align}\label{mh}
D_{x,s,t}=h_{x,s}-h_{x,t}.
\end{align}
\end{itemize}
\end{remark}
A reproducing kernel Hilbert $C^*$-module $\HH$ is relative reproducing kernel since it is enough to set $D_{x,s,t}:=K_{x,s}-K_{x,t}$; compare it with \eqref{mh}. The converse is not true. For a counterexample in the context of scalar-valued relative reproducing kernel Hilbert spaces, see \cite[\S 4]{ALP}. However, as the following proposition shows, the difference is due to existence of a possibly unbounded linear map. To see this, compare \eqref{rrk} with \eqref{rk}.

The next immediate result is an analogous version of \cite[Theorem 2.5]{ALP} from the framework of Hilbert spaces to a more general setting.

\begin{proposition} \label{iff}
A Hilbert $C^*$-module $\HH$ of $\E$-valued functions on a set $S$ is a relative reproducing kernel Hilbert $C^*$-module if and only if there are a map $h: \E\times S \to \HH,\,\,(x,s)\mapsto h_{x,s}$ and a possibly unbounded $\A$-linear map $\varphi: \E\times \HH\to \A$ such that 
\[\langle f(s),x\rangle=\langle f,h_{x,s}\rangle +\varphi(x,f), \qquad f\in\HH, x\in\E, s\in S.\]
\end{proposition}
\begin{proof}
($\Longleftarrow$) It is enough to put $D_{x,s,t}:=h_{x,s}-h_{x,t}$ to reach \eqref{rrk}.\\
($\Longrightarrow$) Fix $s_0\in S$, and set $h_{x,s}:=D_{x,s,s_0}$. Then for all $f\in\HH$, $x\in\E$, and $s\in S$, we have $\langle f(s),x\rangle=\langle f,h_{x,s}\rangle +\varphi(x,f)$, 
where $\varphi: \E\times \HH\to \A$ is defined by $\varphi(x,f)=\langle f(s_0),x\rangle$. 
\end{proof}
\begin{corollary}
Let $\HH$ be a relative reproducing kernel Hilbert $C^*$-module of functions $f: S\to\E$. If there is some $s_0\in S$ such that $f(s_0)=0$ for all $f\in\HH$, then $\HH$ includes a reproducing kernel Hilbert $C^*$-submodule.
\end{corollary}
\begin{proof}
With the notation of Proposition \ref{iff}, we have
\[\langle f(s),x\rangle=\langle f,h_{x,s}\rangle,\quad f\in \HH, s\in S, x\in \E.\]
Let $\HH_0$ be the closed $\A$-linear span of $\{h_{x,s}: s\in S, x\in\E\}$ (see Remark \ref{sum}(i)). Employing Theorem \ref{main1}, we derive that $\HH_0$ is a reproducing kernel Hilbert $C^*$-module with the positive definite kernel $K(s,t)(x)=h_{x,s}(t)$. Note that
\[\sum_{i,j=1}^n\langle K(s_i,s_j)(x_i),x_j\rangle=\sum_{i,q=1}^n\langle h_{x_i,s_i}(s_j),x_j\rangle=\sum_{i,j=1}^n\langle h_{x_i,s_i},h_{x_j,s_j}\rangle\geq 0.\]
\end{proof}
The next theorem provides some mild conditions under which a relative reproducing kernel Hilbert $C^*$-module is reproducing kernel; cf. \cite[Corollary 2.5]{ALP}.

If $\A$ is a $W^*$-algebra, then the $\A$-valued inner product $\langle\cdot,\cdot\rangle$ on a Hilbert $\A$-module $\HH$ can be extended to an $\A$-valued inner product on $\E'$ in such a way as to make $\HH'$ into a self-dual Hilbert $\A$-module. In particular, the extended inner product satisfies $\langle \tau,\widehat{x}\rangle=\tau (x)$ for all $ x\in \HH$ and $\tau\in \HH'$; see \cite[Theorem 3.2]{Pa}. The next result reads as follows. Evidently, the proof works if we consider an arbitrary $C^*$-algebra $\A$ and a self-dual Hilbert $\A$-module $\HH=\HH'$.

\begin{theorem}
Let $\HH$ be a Hilbert $W^*$-module over a $W^*$-algebra $\A$ and let $\HH'$ be a relative reproducing kernel Hilbert $C^*$-module on a set $S$ with a relative reproducing kernel $D$. If there is $s_0\in S$ such that for each $x\in \E$, the $\A$-linear map $f\mapsto \langle x,f(s_0)\rangle$ on $\HH'$ is continuous, then $\HH'$ is a reproducing kernel Hilbert $C^*$-module with a kernel $K$. Moreover, $D_{x,s,t}=K_{x,s}-K_{x,t}$, where $K_{x,t}$ is given as in Theorem \ref{main1}.
\end{theorem}
\begin{proof}
It follows from the self-duality of $\HH'$ that for each $x\in X$ there exists a unique element $\psi_{x}\in \HH'$ such that $\langle \psi_{x},f\rangle=\langle x,f(s_0)\rangle$. Consider the kernel $K:S\times S\to \CL(\E)$ defined by 
\[K(s,t)(x)=h_{x,s}(t)+\psi_{x}(t),\]
where $h_{x,s}:=D_{x,s,s_0}$. Then $K_{x,s}(t):=K(s,t)(x)=h_{x,s}(t)+\psi_{x}(t)$. Hence, $K_{x,s}=h_{x,s}+\psi_{x}\in \HH'$.
Furthermore,
\begin{align*}
\langle f,K_{x,s}\rangle &=\langle f,h_{x,s}+\psi_{x}\rangle\\
& =\langle f(s)-f(s_0),x\rangle +\langle f, \psi_{x}\rangle\\
&=\langle f(s),x\rangle-\langle f(s_0),x\rangle+\langle f(s_0),x\rangle\\
&=\langle f(s),x\rangle
\end{align*}
for all $f\in\HH'$, $x\in \E$, and $s\in S$. Now, the part (ii) of Theorem \ref{main1} entails that $\HH'$ is a reproducing kernel Hilbert $C^*$-module.
Furthermore, it follows from \eqref{mh} that
\[D_{x,s,t}=h_{x,s}-h_{x,t}=(h_{x,s}+\psi_x)-(h_{x,t}-\psi_x)=K_{x,s}-K_{x,t}.\]
\end{proof}

The following theorem gives us a relation between conditionally negative definite kernels and reproducing kernel Hilbert $C^*$-modules. We follow an approach, more suitable for our investigation. 

\begin{theorem}\label{main2}
Let $L$ be a conditionally negative definite kernel on a set $S$ into $\CL(\E)$, let $s_0\in S$, and let $K(s,t)$ be defined as in \eqref{lk}.
Then $K$ is a positive definite kernel and there are a reproducing kernel Hilbert $C^*$-module $\HH$ consisting of functions from $S$ into $\E$ and a map $\psi: S\to \CL(\E)$ such that
\begin{align}\label{lll}
L(s,t)=K(s,s)+K(t,t)-2K(s,t)+\psi(s)+\psi(t)^*.
\end{align}
In particular, $\psi$ is identically zero if and only if $L$ is normalized and $L(s,s_0)$ is self-adjoint.
\end{theorem}
\begin{proof}
Theorem \ref{th1} ensures that 
\[K(s,t):=\frac{1}{2}\left[L(s,s_0) + L(s_0,t)-L(s,t)-L(s_0,s_0)\right]\] is a positive definite kernel. Let $\HH$ be the reproducing kernel Hilbert $C^*$-module constructed in Theorem \ref{main1}(ii). Using the notation of Theorem \ref{main1} and the fact that $L(s_0,s_0)$ is self-adjoint, we have
\begin{align}\label{sch}
\langle K(s,s)x,x\rangle &+\langle K(t,t)x,x\rangle-2 \langle K(s,t)x,x\rangle\nonumber\\
&=\langle K(s,s)x,x\rangle +\langle K(t,t)x,x\rangle-2 {\rm Re}\langle K(s,t)x,x\rangle-2i{\rm Im}\langle K(s,t)x,x\rangle\nonumber\\
&=\frac{1}{2}\langle[L(s,s_0)+L(s_0,s)-L(s,s)-L(s_0,s_0)]x,x\rangle\nonumber\\
&\quad+ \frac{1}{2}\langle[L(t,s_0)+L(s_0,t)-L(t,t)-L(s_0,s_0)]x,x\rangle\nonumber\\
&\quad+ \frac{1}{2}\langle[-L(s,s_0)-L(s_0,t)+L(s,t)+L(s_0,s_0)\nonumber\\
&\quad -L(t,s_0) - L(s_0,s)+L(t,s)+L(s_0,s_0)]x,x\rangle\nonumber\\
&\quad-i{\rm Im}\langle[L(s,s_0)+L(s_0,t)-L(s,t)-L(s_0,s_0)]x,x\rangle\nonumber\\
&={\rm Re} \langle L(s,t)x,x\rangle-\frac{1}{2}[\langle L(s,s)x,x\rangle+\langle L(t,t)x,x\rangle]\nonumber\\
&\quad + {\rm Im} \langle L(s,t)x,x\rangle-i{\rm Im}\langle[L(s,s_0)+L(s_0,t)-L(s_0,s_0)]x,x\rangle\nonumber\\
&= \langle L(s,t)x,x\rangle-\frac{1}{2}[\langle L(s,s)x,x\rangle+\langle L(t,t)x,x\rangle]\nonumber\\
&\quad-i{\rm Im}\langle[L(s,s_0)+L(s_0,t)]x,x\rangle\nonumber\\
&= \langle L(s,t)x,x\rangle-\langle \psi(s)x,x\rangle-\langle \psi(t)^*x,x\rangle],
\end{align}
where
\begin{align}\label{psi}
\psi(s)=\frac{1}{2} L(s,s)+i{\rm Im} L(s,s_0). 
\end{align}
Thus we get \eqref{lll} by utilizing \eqref{pos}. The last assertion is concluded from \eqref{psi} and the fact that $L$ is hermitian.
\end{proof}
\begin{corollary}
Let $L$ be a conditionally negative definite kernel on a set $S$ into the self-adjoint part of $\CL(\E)$. Then the kernel $\langle L(s,t)x,x\rangle$ is conditionally negative definite with values in $\A$ for each $x\in\E$.
\end{corollary}
\begin{proof}
The operator $K(s,t):=\frac{1}{2}\left[L(s,s_0) + L(s_0,t)-L(s,t)-L(s_0,s_0)\right]$ is self-adjoint for each $s,t\in S$. It follows from \eqref{lll} that
\begin{align*}
\langle L(s,t)x,x\rangle &=\langle K(s,s)x,x\rangle+\langle K(t,t)x,x\rangle-\langle K(s,t)x,x\rangle-\langle K(t,s)x,x\rangle\\
&\quad +\langle \psi(s)x,x\rangle+\langle \psi(t)^*x,x\rangle\\
&=|K_{x,s}-K_{x,t}|^2 +\left(\psi(s)x,x\rangle +\overline{\langle \psi(t)x,x\rangle}\right) .
\end{align*}
From Examples \ref{ex1}(iii \& iv), we get the assertion since clearly the sum of two conditionally negative definite kernels is a conditionally negative definite one.
\end{proof}

The next result presents Schoenberg’s lemma \cite{SCH} in our setting.

\begin{corollary} 
Let $L$ be a normalized symmetric conditionally negative definite kernel on a set $S$ into $\CL(\E)$. Then there exist a reproducing kernel Hilbert $C^*$-module $\HH$ consisting of functions from $S$ into $\E$ and a map $\theta: \E\times S\to \HH$ such that
\begin{align*}
\langle L(s,t)x,x\rangle=|\theta(x,s)-\theta(x,t)|^2,\quad x\in\E, s\in S.
\end{align*}
\end{corollary}
\begin{proof}
We use the notation of Theorems \ref{main1} and \ref{lll}. First note that $K$ is symmetric since so is $L$. Set $\theta(x,s):=K_{x,s}$ for $x\in\E$ and $s\in S$. Then, equality \eqref{sch} together with $\psi=0$ yield that
\begin{align*}
|\theta(x,s)-\theta(x,t)|^2&=\langle K_{x,s}-K_{x,t}, K_{x,s}-K_{x,t}\rangle\\
&=\langle K(s,s)x,x\rangle +\langle K(t,t)x,x\rangle-2 \langle K(s,t)x,x\rangle\\
&= \langle L(s,t)x,x\rangle.
\end{align*}
\end{proof}

\medskip

\noindent \textit{Conflict of Interest Statement.} The author states that there is no conflict of interest.\\

\noindent\textit{Data Availability Statement.} Data sharing not applicable to this article as no datasets were generated or analysed during the current study.

\medskip


\end{document}